\documentclass[10pt]{article}
\usepackage{lmodern}
\usepackage{amsmath}
\usepackage[T1]{fontenc}
\usepackage[utf8]{inputenc}
\usepackage{authblk}
\usepackage{amsfonts}
\usepackage{graphicx}
\usepackage{rotating}
\usepackage{amssymb}
\usepackage[english]{babel}
\usepackage{color}
\usepackage{amsthm}
\usepackage{graphicx}
\usepackage{tkz-euclide}
\tikzset{elegant/.style={smooth,thick,samples=50,cyan}}
\usepackage{color,tikz}
\usetikzlibrary{patterns}
\usetikzlibrary{decorations.pathreplacing,calc}
\usepackage{rotating}
\usepackage{mathrsfs}
\usepackage{makecell}
\usepackage{microtype}
\usepackage{enumerate}
\usepackage{mathscinet}
\usepackage{array}
\usepackage{multirow}
\usepackage{booktabs}
\usepackage{enumerate}
\usepackage[cal=boondoxo,bb=ams]{mathalfa}
\usepackage{hyperref}
\hypersetup{hidelinks}
\usepackage{titlesec}
\newtheorem{theorem}{Theorem}[section]
\newtheorem{prop}{Proposition}[section]

\newtheorem{remark}{Remark}[section]

\newcommand{\ml}{\mathcal}
\newcommand{\mb}{\mathbb}

\DeclareMathOperator{\intt}{int}
\DeclareMathOperator{\extt}{ext}
\DeclareMathOperator{\bdd}{bdd}

\title{A note on asymptotic profiles for the thermoelastic plate system}
\author[1]{Wenhui Chen\thanks{Wenhui Chen (wenhui.chen.math@gmail.com)}}
\affil[1]{School of Mathematics and Information Science, Guangzhou University, 510006 Guangzhou, China}
\author[2]{Yan Liu\thanks{Yan Liu (ly801221@163.com)}}
\affil[2]{Department of Applied Mathematics, Guangdong University of Finance, 510521 Guangzhou, China}

\date{}
\begin{document}

\maketitle
\begin{abstract}
	\medskip
We investigate the Cauchy problem for the thermoelastic plate system associated with Newton's law of cooling, where optimal growth ($n\leqslant 4$) or decay ($n\geqslant 5$) estimates and asymptotic profiles of solutions for large-time are studied. Especially, the additional lower-order term in the temperature equation weakens decay rates of the vertical displacement, and leads to a new leading term  comparing with the classical thermoelastic plates.\\
	
	
	\noindent\textbf{AMS Classification (2020)} 35G40, 35B40, 35Q79 
\end{abstract}
\fontsize{12}{15}
\selectfont
\setlength{\textwidth}{7.0in}
\section{Introduction}
Thin plates theory arises in practical applications of engineering (e.g. airport runways, raft foundations and road pavements). In the last century, several mathematical models describing thin plate motions in multifarious cases were built, for instance, Mindlin-Timoshenko models, von K\'arm\'an equations and thermoelastic plate systems. 

Let us take into consideration of a homogeneous, elastic and thermally isotropic plate subjecting to a temperature distribution. Combining with the second law of thermodynamics for irreversible process, the monographs \cite{Lagnese-Lions=1988,Lagnese=1989} modeled the well-known thermoelastic plate system with Fourier's law of heat conduction and Newton's law of cooling, namely,
\begin{align}\label{EQ-GENERAL-MASS}
	\begin{cases}
		u_{tt}+\Delta^2 u+\Delta\theta=0,\\
		\theta_t-\Delta\theta+\sigma\theta-\Delta u_t=0.
	\end{cases}
\end{align}
Here, the scalar unknowns $u=u(t,x)$ and $\theta=\theta(t,x)$ denote, respectively, the vertical displacement and  the temperature (relative to some reference temperature). The non-negative constant $\sigma$ in the temperature equation is related to the heat transfer coefficient. For the heuristic derivations of the thermoelastic plates \eqref{EQ-GENERAL-MASS}, we refer interested readers to \cite[Chapter I.6]{Lagnese-Lions=1988}.

There are numerous studies for the thermoelastic plate system \eqref{EQ-GENERAL-MASS} from the communities of PDEs, controllability, inverse problems and dynamical systems in past thirty years (see \cite{Lagnese-Lions=1988,Kim=1992,Shibata=1994,Liu-Renardy=1995,Munoz-Rivera-Racke=1995,Liu-Liu=1997,Liu-Zheng=1997,Lebeau-Zuazua=1998,Denk-Racke=2006,Naito-Shibata=2009,Lasiecka-Wilke=2013,Said-Houari=2013,Racke-Ueda=2016,Racke-Ueda=2017,Denk-Shibata=2017,Bezerra-Carbone-Nascimento=2019} and references therein). Indeed, most of the recent researches investigated the thermoelastic plate system \eqref{EQ-TEP-MASS} with $\sigma=0$. For example, the corresponding Cauchy problem for \eqref{EQ-TEP-MASS} with $\sigma=0$ has been deeply studied in \cite{Said-Houari=2013,Racke-Ueda=2016,Chen-Ikehata=2022-plate}, in which sharp decay properties of an energy term $(u_t,\Delta u,\theta)$ have been found. Among these results, the authors of \cite{Chen-Ikehata=2022-plate} discovered the critical dimension $n=4$ for the vertical displacement to distinguish different large-time behaviors, namely, optimal growth ($n\leqslant 3$), bounded ($n=4$) and decay ($n\geqslant 5$) estimates. Here, the optimality is guaranteed by the same behaviors for upper and lower bounds. Additionally, they introduced the large-time asymptotic profile $\psi=\psi(t,x)$ of solution by
\begin{align}\label{PSI}
\psi(t,x):=\ml{F}_{\xi\to x}^{-1}\left(\frac{1}{|\xi|^2}\left(\mathrm{e}^{-a_0|\xi|^2t}-\cos(a_2|\xi|^2t)\mathrm{e}^{-a_1|\xi|^2t}\right)\right)P_{\Psi_0}+\ml{F}_{\xi\to x}^{-1}\left(\frac{\sin(a_2|\xi|^2t)}{a_2|\xi|^2}\mathrm{e}^{-a_1|\xi|^2t}\right)P_{\Psi_1}
\end{align}
carrying $\Psi_0:=2a_1u_1+\theta_0$ and $\Psi_1:=(a_0^2+a_2^2-a_1^2)u_1+(a_0-a_1)\theta_0$, where the positive constants $a_0,a_1,a_2$ are defined in the second statement of Proposition \ref{PROP-CHARACTERISTIC-ROOTS}. Nevertheless, the thermoelastic plate system \eqref{EQ-TEP-MASS} with $\sigma>0$  in the whole space $\mb{R}^n$ so far did not been explored yet. Concerning the other works on the model \eqref{EQ-GENERAL-MASS} with $\sigma>0$, we refer \cite{Lagnese-Lions=1988,Lagnese=1989} for the exact controllability problem, and \cite{Avalos-Lasiecka=1998,Liu-Zheng=1997} for the exponential stability of semigroups.

In the present work, we consider the corresponding Cauchy problem for the thermoelastic plate system \eqref{EQ-GENERAL-MASS}, namely,
\begin{align}\label{EQ-TEP-MASS}
\begin{cases}
u_{tt}+\Delta^2 u+\Delta\theta=0,\\
\theta_t-\Delta\theta+\sigma\theta-\Delta u_t=0,\\
(u,u_t,\theta)(0,x)=(u_0,u_1,\theta_0)(x),
\end{cases}
\end{align}
with the constant $\sigma>0$, and $(t,x)\in\mb{R}_+\times\mb{R}^n$ for any $n\geqslant 1$. Our main interests are large-time asymptotic behaviors of the vertical displacement influenced by the lower-order term $+\sigma\theta$ in the temperature equation \eqref{EQ-TEP-MASS}$_2$ from Newton's law of cooling. By applying WKB method and Fourier analysis, we derive optimal growth estimates when $n\leqslant 4$ and decay estimates when $n\geqslant 5$ for the vertical displacement in the $L^2$ framework. Among them, as we will state in Remark \ref{Rem-Key} and Table \ref{tab:table1}, the growth rates when $n\leqslant 4$ are the same as those for the pure plate model, but the growth rate when $n=4$ and decay rates when $n\geqslant 5$ are weakened by the lower-order term $+\sigma\theta$ comparing with the result in \cite{Chen-Ikehata=2022-plate}. Additionally, the dominant Fourier multiplier of asymptotic profiles has been greatly changed into 
\begin{align*}
\ml{F}^{-1}_{\xi\to x}\left(\frac{\sin(|\xi|^2t)}{|\xi|^2}\mathrm{e}^{-\frac{1}{2\sigma}|\xi|^4t}\right),
\end{align*}
which causes the weakened effect of decay rates when $n\geqslant 4$. Indeed, the effect of the lower-order term $+\sigma\theta$ in the temperature (parabolic) equation will propagate throughout $(\Delta\theta,-\Delta u_t)^{\mathrm{T}}$ to the plate model. It leads to slower decay rates for $n\geqslant 4$ in comparison with the classical model  with $\sigma=0$.

\medskip
\noindent\textbf{Notations:} Let us take the following zones:
\begin{align*}
	\ml{Z}_{\intt}(\varepsilon_0):=\{|\xi|\leqslant\varepsilon_0\ll1\},\ \ 
	\ml{Z}_{\bdd}(\varepsilon_0,N_0):=\{\varepsilon_0\leqslant |\xi|\leqslant N_0\},\ \  
	\ml{Z}_{\extt}(N_0):=\{ |\xi|\geqslant N_0\gg1\}.
\end{align*}
Additionally, the cut-off functions $\chi_{\intt}(\xi),\chi_{\bdd}(\xi),\chi_{\extt}(\xi)\in \mathcal{C}^{\infty}$ own  supports in their corresponding zones $\ml{Z}_{\intt}(\varepsilon_0)$, $\ml{Z}_{\bdd}(\varepsilon_0/2,2N_0)$ and $\ml{Z}_{\extt}(N_0)$, respectively, so that $\chi_{\bdd}(\xi)=1-\chi_{\intt}(\xi)-\chi_{\extt}(\xi)$. The relation $f\lesssim g$ means that there exists a positive constant $C$ fulfilling $f\leqslant Cg$, which may be changed in different lines, analogously, for $f\gtrsim g$. Furthermore, the asymptotic behavior $f\simeq  g$ holds if and only if $g\lesssim f\lesssim g$. We denote $\langle\cdot,\cdot\rangle$ by the inner product in Euclidean space. Let us recall the weighted $L^1$ space as follows:
\begin{align*}
	L^{1,1}:=\left\{f\in L^1 \ \big|\ \|f\|_{L^{1,1}}:=\int_{\mb{R}^n}(1+|x|)|f(x)|\mathrm{d}x<\infty \right\}.
\end{align*}
The means of a summable function $f$ are denoted by $P_f:=\int_{\mb{R}^n}f(x)\mathrm{d}x$ as well as $M_f:=\int_{\mb{R}^n}xf(x)\mathrm{d}x$. To complete the introduction, we take the time-dependent functions
\begin{align}\label{Decay-fun}
	\ml{D}_n(t):=\begin{cases}
		t^{1-\frac{n}{4}}&\mbox{if}\ \ n\leqslant 3,\\
		\sqrt{\ln t}&\mbox{if}\ \ n=4,\\
		t^{\frac{1}{2}-\frac{n}{8}}&\mbox{if}\ \ n\geqslant 5,
	\end{cases}
\ \ \mbox{and}\ \
\ml{B}_n(t):=\begin{cases}
	t^{\frac{1}{4}}&\mbox{if}\ \ n=1,\\
	\sqrt{\ln t}&\mbox{if}\ \ n=2,\\
	t^{\frac{1}{4}-\frac{n}{8}}&\mbox{if}\ \ n\geqslant 3,
\end{cases} 
\end{align}
to be the growth or decay coefficients later.

\section{Main results}
Our first result contributes to optimal estimates for the vertical displacement in the $L^2$ norm. Especially, the solution grows to infinite polynomially (lower-dimension $n=1,2,3$) or  logarithmically (critical-dimension $n=4$) as $t\to\infty$.
\begin{theorem}\label{THM-OPTIMAL-DECAY}
	Let us consider the thermoelastic plate system \eqref{EQ-TEP-MASS} with $\sigma>0$ and carrying initial datum $u_0,u_1,\theta_0\in L^2\cap L^1$. Then, its vertical displacement fulfills the following optimal estimates:
	\begin{align}\label{Optimal-Est}
		\ml{D}_n(t)|P_{u_1}|\lesssim \|u(t,\cdot)\|_{L^2}\lesssim\ml{D}_n(t)\|(u_0,u_1,\theta_0)\|_{(L^2\cap L^1)^3}
	\end{align}
for $t\gg1$, where the time-dependent coefficient $\ml{D}_n(t)$ was defined in \eqref{Decay-fun}. Namely, if $|P_{u_1}|\neq0$, then the optimal estimates $\|u(t,\cdot)\|_{L^2}\simeq\ml{D}_n(t)$ hold for any $n\geqslant 1$ as $t\gg1$.
\end{theorem}
\begin{remark}\label{Rem-Key}
	Let us recall the large-time behavior of the pure plate model
	\begin{align}\label{EQ-PLATE}
		\begin{cases}
		w_{tt}+\Delta^2 w=0,\\
		(w,w_t)(0,x)=(w_0,w_1)(x),
		\end{cases}
	\end{align}
with $(t,x)\in\mb{R}_+\times\mb{R}^n$. The author of \cite{Ikehata=2021-Dec} derived optimal growth estimates
\begin{align}\label{decay-w}
\|w(t,\cdot)\|_{L^2}^2\simeq\begin{cases}
	t^{2-\frac{n}{2}}&\mbox{if}\ \ n\leqslant 3,\\
	\ln t&\mbox{if}\ \ n=4,
\end{cases}
\end{align}
for $t\gg1$, where initial datum are taken in $L^2\cap L^1$. For another, the authors of \cite{Chen-Ikehata=2022-plate} derived the optimal estimates for the classical thermoelastic plate system (or the model \eqref{EQ-TEP-MASS} with $\sigma=0$) as follows:
\begin{align}\label{EQ-TEP-CLASSICAL}
	\begin{cases}
		v_{tt}+\Delta^2 v+\Delta\theta=0,\\
		\theta_t-\Delta\theta-\Delta v_t=0,\\
		(v,v_t,\theta)(0,x)=(v_0,v_1,\theta_0)(x).
	\end{cases}
\end{align}
To be specific, concerning all dimensions $n\geqslant 1$, they got
\begin{align}\label{decay-v}
\|v(t,\cdot)\|_{L^2}^2\simeq t^{2-\frac{n}{2}}
\end{align}
for $t\gg1$. Comparing with  optimal estimates \eqref{EQ-PLATE} in Theorem \ref{THM-OPTIMAL-DECAY}, \eqref{decay-w} in \cite{Ikehata=2021-Dec}, and \eqref{decay-v} in \cite{Chen-Ikehata=2022-plate}, we claim the following statements (see Table \ref{tab:table1} as the supplementary):
\begin{itemize}
	\item In the lower-dimensions $n=1,2,3$, the growth rates of the  models \eqref{EQ-TEP-MASS} and \eqref{EQ-TEP-CLASSICAL} are the same as those for the pure plate model \eqref{EQ-PLATE}. That is to say that the plate equation plays the decisive role in the general thermoelastic plate system \eqref{EQ-TEP-MASS} for any $\sigma\geqslant 0$.
	\item In the critical-dimension $n=4$, the lower-order term $+\sigma\theta$ subdues the estimates in the classical thermoelastic plates \eqref{EQ-TEP-CLASSICAL} so that the plate equation has the decisive influence again with the same growth rate $\log t$ as the one in the plate model \eqref{EQ-PLATE}.
	\item In the higher-dimensions $n\geqslant 5$, the lower-order term $+\sigma\theta$ weakens the decay rates from $t^{2-\frac{n}{2}}$ to $t^{1-\frac{n}{4}}$. This weakened effect is originated  from the lower-order term in the temperature equation. In some sense, the lower-order term propagates via the coupling $(\Delta \theta,-\Delta u_t)^{\mathrm{T}}$. This effect is quite different from the evolution equations with the mass term, e.g. heat equations with mass \cite[Chapter 12.2]{Ebert-Reissig-book}, Klein-Gordon equation \cite[Chapter 11.3.4]{Ebert-Reissig-book}, damped waves with mass \cite{N-Palmieri-R}, and strongly damped waves with mass \cite{Dabbicco-Ikehata=2019}. To the best of authors' knowledge, it seems to be the first example that an addition of lower-order term (sometimes we call it mass term) causes weakened dissipative properties.
\end{itemize}
\begin{table}[h!]
	\begin{center}
		\caption{Influence from the plate model, Fourier's law and the lower-order term}
		\label{tab:table1}
		\begin{tabular}{cccc} 
			\toprule
			\multirow{2}*{Dimensions} & $n\leqslant 3$ & $n=4$ & $n\geqslant 5$\\
			& (Lower-dimensions) & (Critical-dimension) & (Higher-dimensions)  \\
			\midrule
			\multirow{2}*{Pure plate} & \multirow{2}*{$t^{2-\frac{n}{2}}$} & \multirow{2}*{$\log t$} & \multirow{2}*{--} \\
			& & &\\  
			\midrule
			Thermoelastic plates  & \multirow{2}*{$t^{2-\frac{n}{2}}$} & \multirow{2}*{$1$} & \multirow{2}*{$t^{2-\frac{n}{2}}$} \\
			without lower-order term ($\sigma=0$) & & &\\
			\midrule
			Thermoelastic plates  & \multirow{2}*{$t^{2-\frac{n}{2}}$} & \multirow{2}*{$\log t$} & \multirow{2}*{$t^{1-\frac{n}{4}}$}\\
			with lower-order term ($\sigma>0$)& & &\\
			\bottomrule
			\multicolumn{4}{l}{\emph{$*$Optimal estimates of the vertical displacement in the $L^2$ framework.}} 
		\end{tabular}
	\end{center}
\end{table}
\end{remark}
\begin{remark}
In our result for optimal estimates, we just require additional $L^1$ regularity for initial datum rather than $L^{1,1}$ regularity in \cite{Ikehata=2021-Dec,Chen-Ikehata=2022-plate}. Indeed, the weighted $L^1$ assumption for the plate models \cite{Ikehata=2021-Dec,Chen-Ikehata=2022-plate} can be relaxed by $L^1$ assumption by following our approach.
\end{remark}

Before stating the result concerning asymptotic profiles, let us take
\begin{align}\label{Profile}
	\varphi(t,x):=J_0(t,x)P_{u_1}+\left\langle\nabla J_0(t,x), M_{u_1}\right\rangle+H(t,x)P_{u_1}+J_1(t,x)P_{u_0}-\sigma^{-1}\Delta J_0(t,x)P_{\theta_0},
\end{align}
where some functions originated from higher-order diffusion-plates are chosen by
\begin{align*}
J_0(t,x):=\ml{F}^{-1}_{\xi\to x}\left(\frac{\sin(|\xi|^2t)}{|\xi|^2}\mathrm{e}^{-\frac{1}{2\sigma}|\xi|^4t}\right),\ \ J_1(t,x):=\ml{F}^{-1}_{\xi\to x}\left(\cos(|\xi|^2t)\mathrm{e}^{-\frac{1}{2\sigma}|\xi|^4t}\right),
\end{align*}
as well as
\begin{align*}
	H(t,x):=\frac{t}{8\sigma^2}\ml{F}^{-1}_{\xi\to x}\left[\left((2\sigma+1)\cos(|\xi|^2t)+4\sin(|\xi|^2t)\right)|\xi|^4\mathrm{e}^{-\frac{1}{2\sigma}|\xi|^4t}\right].
\end{align*}
 We now may show the asymptotic profile of solution, where the time-dependent coefficient of optimal estimates \eqref{Optimal-Est} has been improved as $t\gg1$ when we subtract the profile $\varphi=\varphi(t,x)$. Particularly, in the higher-dimensional case $n\geqslant 3$, we claim $u(t,\cdot)\to\varphi(t,\cdot)$ as $t\to\infty$ in the $L^2$ framework.
\begin{theorem}\label{THM-OPTIMAL-LEADING}
		Let us consider the thermoelastic plate system \eqref{EQ-TEP-MASS} with $\sigma>0$ and carrying initial datum $u_0,u_1,\theta_0\in L^2\cap L^{1,1}$. Then, its vertical displacement fulfills the following refined estimates:
	\begin{align*}
\|u(t,\cdot)-\varphi(t,\cdot)\|_{L^2}=o\big(\ml{B}_n(t)\big)
	\end{align*}
	for $t\gg1$, where the time-dependent coefficient $\ml{B}_n(t)$ was defined in \eqref{Decay-fun}.
\end{theorem}
\begin{remark}
The lower-order term $+\sigma\theta$ enables the asymptotic profiles of the thermoelastic plate system \eqref{EQ-TEP-MASS} to change $\psi=\psi(t,x)$ in \eqref{PSI} of the model with $\sigma=0$ into $\varphi=\varphi(t,x)$ in \eqref{Profile} of the model with $\sigma>0$. The vital difference is the power of exponential function in the Fourier multipliers.
\end{remark}
\begin{remark}
Physically, the three-dimensional thermoelastic plate system is more important than other cases because of its applications in practice. If we simply consider $\varphi_{\mathrm{sim}}(t,x)=J_0(t,x)P_{u_1}$ to be the asymptotic profile, then it is not difficult to prove $\|u(t,\cdot)-\varphi_{\mathrm{sim}}(t,\cdot)\|_{L^2}=o(\ml{D}_n(t))$ as $t\gg1$, which does not decay in $\mb{R}^3$ as $t\to\infty$. For this purpose, we construct the higher-order profile $\varphi(t,x)$ instead of $\varphi_{\mathrm{sim}}(t,x)$.
\end{remark}

\section{Asymptotic behaviors of solutions}
\subsection{Reduction procedure and characteristic roots}	
To begin with this section, strongly motivated by the recent work \cite{Chen-Ikehata=2022-plate}, we employ the reduction procedure with respect to the vertical displacement. Due to $(\partial_t-\Delta+\sigma\ml{I})\theta=\Delta u_t$ equipping the identity operator $\ml{I}$, we act this diffusion operator on \eqref{EQ-TEP-MASS}$_1$ to arrive at the third-order (in time) evolution equation as follows:
\begin{align}\label{EQ-THIRD-ORDER}
	\begin{cases}
	u_{ttt}+(\sigma\ml{I}-\Delta)u_{tt}+2\Delta^2u_t+(\sigma\ml{I}-\Delta)\Delta^2u=0,\\
	(u,u_t,u_{tt})(0,x)=(u_0,u_1,u_2)(x),
	\end{cases}
\end{align}
with $(t,x)\in\mb{R}_+\times\mb{R}^n$, where the third data is determined by $u_2:=-\Delta^2u_0-\Delta\theta_0$. An application of the partial Fourier transform to the model \eqref{EQ-THIRD-ORDER} with respect to spatial variables yields
\begin{align}\label{EQ-FOURIER}
\begin{cases}
\widehat{u}_{ttt}+(\sigma+|\xi|^2)\widehat{u}_{tt}+2|\xi|^4\widehat{u}_t+(\sigma+|\xi|^2)|\xi|^4\widehat{u}=0,\\
(\widehat{u},\widehat{u}_t,\widehat{u}_{tt})(0,\xi)=(\widehat{u}_0,\widehat{u}_1,\widehat{u}_2)(\xi),
\end{cases}
\end{align}
with  $(t,\xi)\in\mb{R}_+\times\mb{R}^n$. Note that $\widehat{u}_2=-|\xi|^4\widehat{u}_0+|\xi|^2\widehat{\theta}_0$. Its characteristic equation is given by
\begin{align}\label{CUBIC}
	\lambda^3+(\sigma+|\xi|^2)\lambda^2+2|\xi|^4\lambda+(\sigma+|\xi|^2)|\xi|^4=0.
\end{align}
This cubic owns the strictly negative discriminant
\begin{align*}
	\triangle_{\mathrm{dis}}=-4\left[(\sigma+|\xi|^2)^2|\xi|^2-2\sqrt{2}|\xi|^6\right]^2-(16\sqrt{2}-13)(\sigma+|\xi|^2)^2|\xi|^8<0.
\end{align*}
For this reason, the last cubic \eqref{CUBIC} has one real root $\lambda_1$ and two complex conjugate roots $\lambda_{2/3}=\lambda_{\mathrm{R}}\pm i\lambda_{\mathrm{I}}$ carrying $\lambda_{\mathrm{R}},\lambda_{\mathrm{I}}\in\mb{R}$. Different from the homogeneous characteristic equation of the model with $\sigma=0$ in \cite{Chen-Ikehata=2022-plate}, we now cannot trivially stick down the explicit roots. 

According to the size of frequencies $|\xi|$, one may separate the discussion into three zones.  To be specific, we employ Taylor-like asymptotic expansions for $\xi\in\ml{Z}_{\intt}(\varepsilon_0)\cup\ml{Z}_{\extt}(N_0)$, and a contradiction argument combined with continuity of the roots for $\xi\in\ml{Z}_{\bdd}(\varepsilon_0,N_0)$. By proceeding lengthy but straightforward computations, we state the next behaviors for the roots.
\begin{prop}\label{PROP-CHARACTERISTIC-ROOTS}
The characteristic roots $\lambda_j=\lambda_j(|\xi|)$ with $j=1,2,3$ to the cubic equation \eqref{CUBIC} can be expanded by the next forms.
\begin{itemize}
	\item Concerning $\xi\in\ml{Z}_{\intt}(\varepsilon_0)$, three characteristic roots behave as
	\begin{align*}
	\lambda_1&=-\sigma+\frac{\sigma}{2-3\sigma}|\xi|^2+\ml{O}(|\xi|^4),\\
	\lambda_{2/3}&=\pm i|\xi|^2-\frac{1}{2\sigma}|\xi|^4+\left(\frac{1}{2\sigma^2}\pm\frac{(2\sigma+1)i}{8\sigma^2}\right)|\xi|^6+\ml{O}(|\xi|^8),
	\end{align*}
namely, $\lambda_{\mathrm{R}}=-\frac{1}{2\sigma}|\xi|^4+\frac{1}{2\sigma^2}|\xi|^6+\ml{O}(|\xi|^8)$ and $\lambda_{\mathrm{I}}=|\xi|^2+\frac{2\sigma+1}{8\sigma^2}|\xi|^6+\ml{O}(|\xi|^8)$.
	\item Concerning $\xi\in\ml{Z}_{\extt}(N_0)$, three characteristic roots behave as
	\begin{align*}
	\lambda_1&=-a_0|\xi|^2+\ml{O}(1),\\
	\lambda_{2/3}&=-a_1|\xi|^2\mp ia_2|\xi|^2+\ml{O}(1),
	\end{align*}
where $a_0:=\frac{1+\alpha_-}{3}\approx 0.57$, $a_1:=\frac{2-\alpha_-}{6}\approx 0.22$ and $a_2:=\frac{\sqrt{3}\alpha_+}{6}\approx 1.31$ carrying the constants
\begin{align*}
	\alpha_{\pm}:=\sqrt[3]{\frac{1}{2}(3\sqrt{69}+11)}\pm \sqrt[3]{\frac{1}{2}(3\sqrt{69}-11)},
\end{align*}
namely, $\lambda_{\mathrm{R}}=-a_1|\xi|^2+\ml{O}(1)$ and $\lambda_{\mathrm{I}}=-a_2|\xi|^2+\ml{O}(1)$.
	\item Concerning $\xi\in\ml{Z}_{\bdd}(\varepsilon_0,N_0)$, three characteristic roots fulfill $\Re\lambda_j<0$ for all $j=1,2,3$.
\end{itemize}
\end{prop}
\begin{remark}
	In the small frequencies portion of Proposition \ref{PROP-CHARACTERISTIC-ROOTS}, we not only obtain pairwise distinct characteristic roots with negative real parts, but also investigate some higher-order terms, i.e. $|\xi|^2$-term in $\lambda_1$ and $|\xi|^6$-terms in $\lambda_{2/3}$ when $\xi\in\ml{Z}_{\intt}(\varepsilon_0)$. These non-vanishing higher-order terms will contribute to further expansions of solutions in the Fourier space.
\end{remark}

\subsection{Pointwise estimates and auxiliary functions in the Fourier space}
Throughout this subsection, we consider $\xi\in\ml{Z}_{\intt}(\varepsilon_0)\cup\ml{Z}_{\extt}(N_0)$ only because of exponential decay estimates for $\xi\in\ml{Z}_{\bdd}(\varepsilon_0,N_0)$. According to the representation of solution to the third-order model \eqref{EQ-FOURIER} and the expression of the last data $\widehat{u}_2=-|\xi|^4\widehat{u}_0+|\xi|^2\widehat{\theta}_0$, we may arrive at
\begin{align*}
	\widehat{u}
	&=\frac{(|\xi|^4-\lambda_{\mathrm{I}}^2-\lambda_{\mathrm{R}}^2)\widehat{u}_0+2\lambda_{\mathrm{R}}\widehat{u}_1-|\xi|^2\widehat{\theta}_0}{2\lambda_{\mathrm{R}}\lambda_1-\lambda_{\mathrm{I}}^2-\lambda_{\mathrm{R}}^2-\lambda_1^2}\mathrm{e}^{\lambda_1t}+\frac{(2\lambda_{\mathrm{R}}\lambda_1-\lambda_1^2-|\xi|^4)\widehat{u}_0-2\lambda_{\mathrm{R}}\widehat{u}_1+|\xi|^2\widehat{\theta}_0}{2\lambda_{\mathrm{R}}\lambda_1-\lambda_{\mathrm{I}}^2-\lambda_{\mathrm{R}}^2-\lambda_1^2}\cos(\lambda_{\mathrm{I}}t)\mathrm{e}^{\lambda_{\mathrm{R}}t}\\
	&\quad+\frac{[\lambda_1(\lambda_{\mathrm{R}}\lambda_1+\lambda_{\mathrm{I}}^2-\lambda_{\mathrm{R}}^2)+|\xi|^4(\lambda_{\mathrm{R}}-\lambda_1)]\widehat{u}_0+(\lambda_{\mathrm{R}}^2-\lambda_{\mathrm{I}}^2-\lambda_1^2)\widehat{u}_1-|\xi|^2(\lambda_{\mathrm{R}}-\lambda_1)\widehat{\theta}_0}{\lambda_{\mathrm{I}}(2\lambda_{\mathrm{R}}\lambda_1-\lambda_{\mathrm{I}}^2-\lambda_{\mathrm{R}}^2-\lambda_1^2)}\sin(\lambda_{\mathrm{I}}t)\mathrm{e}^{\lambda_{\mathrm{R}}t},
\end{align*}
whose idea should be traced back to \cite[Subsection 2.3]{Chen-Takeda=2022-JMGT}.

To analyze asymptotic behaviors of solution, we firstly set $\widehat{g}_j=\widehat{g}_j(t,\xi)$ with $j=1,2$ such that
\begin{align*}
	\widehat{g}_1&:=\frac{-\lambda_1^2\sin(\lambda_{\mathrm{I}}t)\mathrm{e}^{\lambda_{\mathrm{R}}t}}{\lambda_{\mathrm{I}}(2\lambda_{\mathrm{R}}\lambda_1-\lambda_{\mathrm{I}}^2-\lambda_{\mathrm{R}}^2-\lambda_1^2)}\widehat{u}_1,\\
	\widehat{g}_2&:=\frac{-\lambda_1^2\cos(\lambda_{\mathrm{I}}t)\mathrm{e}^{\lambda_{\mathrm{R}}t}}{2\lambda_{\mathrm{R}}\lambda_1-\lambda_{\mathrm{I}}^2-\lambda_{\mathrm{R}}^2-\lambda_1^2}\widehat{u}_0+\frac{|\xi|^2\lambda_1\sin(\lambda_{\mathrm{I}}t)\mathrm{e}^{\lambda_{\mathrm{R}}t}}{\lambda_{\mathrm{I}}(2\lambda_{\mathrm{R}}\lambda_1-\lambda_{\mathrm{I}}^2-\lambda_{\mathrm{R}}^2-\lambda_1^2)}\widehat{\theta}_0,
\end{align*}
whose origins are extractions of leading terms of $\widehat{u}$ for small frequencies. Applying the asymptotic expansions derived in Proposition \ref{PROP-CHARACTERISTIC-ROOTS}, the next error estimate holds:
\begin{align}\label{star-1}
\chi_{\intt}(\xi)|\widehat{u}-\widehat{g}_1-\widehat{g}_2|&\lesssim\chi_{\intt}(\xi)\left(\mathrm{e}^{-ct}+|\cos(|\xi|^2t)|\mathrm{e}^{-c|\xi|^4t}\right)\left(|\xi|^6|\widehat{u}_0|+|\xi|^4|\widehat{u}_1|+|\xi|^2|\widehat{\theta}_0|\right)\notag\\
&\quad+\chi_{\intt}(\xi)|\sin(|\xi|^2t)|\mathrm{e}^{-c|\xi|^4t}\left(|\xi|^2|\widehat{u}_0|+|\xi|^2|\widehat{u}_1|+|\xi|^4|\widehat{\theta}_0|\right)\notag\\
&\lesssim \chi_{\intt}(\xi)|\xi|^2\mathrm{e}^{-c|\xi|^4t}\left(|\widehat{u}_0|+|\widehat{u}_1|+|\widehat{\theta}_0|\right), 
\end{align}
where we took advantages of  $|\xi|^4-\lambda_{\mathrm{I}}^2=\ml{O}(|\xi|^8)$ as well as $2\lambda_{\mathrm{R}}\lambda_1-|\xi|^4=\ml{O}(|\xi|^6)$ for $\xi\in\ml{Z}_{\intt}(\varepsilon_0)$. Due to the estimates that
\begin{align*}
\chi_{\intt}(\xi)|\widehat{g}_1|&\lesssim\chi_{\intt}(\xi)\frac{|\sin(|\xi|^2t)|}{|\xi|^2}\mathrm{e}^{-c|\xi|^4t}|\widehat{u}_1|,\\
\chi_{\intt}(\xi)|\widehat{g}_2|&\lesssim\chi_{\intt}(\xi)\mathrm{e}^{-c|\xi|^4t}\left(|\widehat{u}_0|+|\widehat{\theta}_0|\right),
\end{align*}
we use the triangle inequality to obtain
\begin{align}
\chi_{\intt}(\xi)|\widehat{u}-\widehat{g}_1|&\lesssim \chi_{\intt}(\xi)\mathrm{e}^{-c|\xi|^4t}\left(|\widehat{u}_0|+|\xi|^2|\widehat{u}_1|+|\widehat{\theta}_0|\right),\label{Est-05}\\
\chi_{\intt}(\xi)|\widehat{u}|&\lesssim \chi_{\intt}(\xi)\mathrm{e}^{-c|\xi|^4t}\left(|\widehat{u}_0|+\frac{|\sin(|\xi|^2t)|}{|\xi|^2}|\widehat{u}_1|+|\widehat{\theta}_0|\right).\label{Est-04}
\end{align}

Let us introduce some approximations $\widehat{J}_j=\widehat{J}_j(t,|\xi|)$ with $j=0,1$ such that
\begin{align*}
\widehat{J}_0:=\frac{\sin(|\xi|^2t)}{|\xi|^2}\mathrm{e}^{-\frac{1}{2\sigma}|\xi|^4t}\ \ \mbox{and}\ \ \widehat{J}_1:=\cos(|\xi|^2t)\mathrm{e}^{-\frac{1}{2\sigma}|\xi|^4t}.
\end{align*}
They are the Fourier transformations of higher-order diffusion-plates. Then, we are able to claim the refined estimates by subtracting some approximated functions.
\begin{prop}\label{PROP-APPROXIMATION}
Concerning $\xi\in\ml{Z}_{\intt}(\varepsilon_0)$, the following refined estimates hold:
\begin{align}
\chi_{\intt}(\xi)|\widehat{g}_1-\widehat{J}_0\widehat{u}_1|&\lesssim\chi_{\intt}(\xi)\mathrm{e}^{-c|\xi|^4t}|\widehat{u}_1|,\label{Est-01}\\
\chi_{\intt}(\xi)|\widehat{g}_1-(\widehat{J}_0+\widehat{H}_0+\widehat{H}_1)\widehat{u}_1|&\lesssim\chi_{\intt}(\xi)|\xi|^2\mathrm{e}^{-c|\xi|^4t}|\widehat{u}_1|,\label{Est-02}\\
\chi_{\intt}(\xi)|\widehat{g}_2-\widehat{J}_1\widehat{u}_0-\sigma^{-1}|\xi|^2\widehat{J}_0\widehat{\theta}_0|&\lesssim\chi_{\intt}(\xi)|\xi|^2\mathrm{e}^{-c|\xi|^4t}\left(|\widehat{u}_0|+|\widehat{\theta}_0|\right),\label{Est-03}
\end{align}
where the auxiliary functions  $\widehat{H}_j=\widehat{H}_j(t,|\xi|)$ with $j=0,1$ are defined by
\begin{align*}
\widehat{H}_0:=\frac{2\sigma+1}{8\sigma^2}|\xi|^4t\cos(|\xi|^2t)\mathrm{e}^{-\frac{1}{2\sigma}|\xi|^4t}\ \ \mbox{and}\ \ 
\widehat{H}_1:=\frac{1}{2\sigma^2}|\xi|^4t\sin(|\xi|^2t)\mathrm{e}^{-\frac{1}{2\sigma}|\xi|^4t}.
\end{align*}
\end{prop}
\begin{proof}
A direct subtraction associated with suitable decompositions implies
\begin{align*}
\widehat{g}_1-(\widehat{J}_0+\widehat{H}_0+\widehat{H}_1)\widehat{u}_1=\sum\limits_{j=0}^3\widehat{E}_j\widehat{u}_1,
\end{align*}
where the error terms are
\begin{align*}
\widehat{E}_0&:=\frac{-\lambda_1^2\sin(\lambda_{\mathrm{I}}t)\mathrm{e}^{\lambda_{\mathrm{R}}t}}{\lambda_{\mathrm{I}}(2\lambda_{\mathrm{R}}\lambda_1-\lambda_{\mathrm{I}}^2-\lambda_{\mathrm{R}}^2-\lambda_1^2)}-\frac{1}{|\xi|^2}\sin(\lambda_{\mathrm{I}}t)\mathrm{e}^{\lambda_{\mathrm{R}}t},\\
\widehat{E}_1&:=\frac{1}{|\xi|^2}\left(\sin(\lambda_{\mathrm{I}}t)-\sin(|\xi|^2t)-\frac{2\sigma+1}{8\sigma^2}|\xi|^6t\cos(|\xi|^2t)\right)\mathrm{e}^{\lambda_{\mathrm{R}}t},\\
\widehat{E}_2&:=\frac{1}{|\xi|^2}\sin(|\xi|^2t)\left(\mathrm{e}^{\lambda_{\mathrm{R}}t}-\mathrm{e}^{-\frac{1}{2\sigma}|\xi|^4t}-\frac{1}{2\sigma^2}|\xi|^6t\mathrm{e}^{-\frac{1}{2\sigma}|\xi|^4t}\right),\\
\widehat{E}_3&:=\frac{2\sigma+1}{8\sigma^2}|\xi|^4t\cos(|\xi|^2t)\left(\mathrm{e}^{\lambda_{\mathrm{R}}t}-\mathrm{e}^{-\frac{1}{2\sigma}|\xi|^4t}\right).
\end{align*}
For one thing, the asymptotic behaviors stated in Proposition \ref{PROP-CHARACTERISTIC-ROOTS} leads to
\begin{align*}
\chi_{\intt}(\xi)|\widehat{E}_0|&=\chi_{\intt}(\xi)\left|\frac{\lambda_1^2(\lambda_{\mathrm{I}}-|\xi|^2)-\lambda_{\mathrm{I}}(2\lambda_{\mathrm{R}}\lambda_1-\lambda_{\mathrm{I}}^2-\lambda_{\mathrm{R}}^2)}{\lambda_{\mathrm{I}}(2\lambda_{\mathrm{R}}\lambda_1-\lambda_{\mathrm{I}}^2-\lambda_{\mathrm{R}}^2-\lambda_1^2)|\xi|^2} \right||\sin(\lambda_{\mathrm{I}}t)|\mathrm{e}^{\lambda_{\mathrm{R}}t}\\
&\lesssim\chi_{\intt}(\xi)|\xi|^2\mathrm{e}^{-c|\xi|^4t},
\end{align*}
because of $\lambda_{\mathrm{I}}-|\xi|^2=\ml{O}(|\xi|^6)$ for $\xi\in\ml{Z}_{\intt}(\varepsilon_0)$. For another, with the help of Taylor's expansions as $|\xi|\ll 1$, we notice
\begin{align*}
	\sin(\lambda_{\mathrm{I}}t)&=\sin(|\xi|^2t)+\frac{2\sigma+1}{8\sigma^2}|\xi|^6t\cos(|\xi|^2t)+\ml{O}(|\xi|^{12})t^2,\\
	\mathrm{e}^{\lambda_{\mathrm{R}}t}&=\mathrm{e}^{-\frac{1}{2\sigma}|\xi|^4t}+\frac{1}{2\sigma^2}|\xi|^6t\mathrm{e}^{-\frac{1}{2\sigma}|\xi|^4t}+\ml{O}(|\xi|^{12})t^2\mathrm{e}^{-\frac{1}{2\sigma}|\xi|^4t}.
\end{align*}
As a consequence,
\begin{align*}
\chi_{\intt}(\xi)\left(|\widehat{E}_1|+|\widehat{E}_2|+|\widehat{E}_3|\right)\lesssim\chi_{\intt}(\xi)|\xi|^{10}t^2\mathrm{e}^{-c|\xi|^4t}\lesssim\chi_{\intt}(\xi)|\xi|^2\mathrm{e}^{-c|\xi|^4t}.
\end{align*}
Summarizing the last estimates, we conclude
\begin{align*}
\chi_{\intt}(\xi)|\widehat{g}_1-(\widehat{J}_0+\widehat{H}_0+\widehat{H}_1)\widehat{u}_1|\lesssim \chi_{\intt}(\xi)\sum\limits_{j=0}^3|\widehat{E}_j||\widehat{u}_1|\lesssim\chi_{\intt}(\xi)|\xi|^2\mathrm{e}^{-c|\xi|^4t}|\widehat{u}_1|,
\end{align*}
which completes the proof of \eqref{Est-02}. Additionally, the use of the triangle inequality associated with
\begin{align*}
	\chi_{\intt}(\xi)\left(|\widehat{H}_0|+|\widehat{H}_1|\right)\lesssim\chi_{\intt}(\xi)\mathrm{e}^{-c|\xi|^4t},
\end{align*}
we claim the desired estimate \eqref{Est-01}. By a similar approach to the above, we may prove \eqref{Est-03} easily.
\end{proof}

Eventually, we propose pointwise estimates under bounded and large frequencies such that
	\begin{align}\label{PROP-LARGE-BDD}
	\big(\chi_{\bdd}(\xi)+\chi_{\extt}(\xi)\big)|\hat{u}|\lesssim\big(\chi_{\bdd}(\xi)+\chi_{\extt}(\xi)\big)\mathrm{e}^{-ct}\left(|\widehat{u}_0|+\langle\xi\rangle^{-2}|\widehat{u}_1|+\langle\xi\rangle^{-2}|\widehat{\theta}_0|\right),
\end{align}
 whose proof is standard basing on Proposition \ref{PROP-CHARACTERISTIC-ROOTS}. One may see the second estimate in \cite[Proposition 3.1]{Chen-Ikehata=2022-plate}. They will not influence on large-time behaviors of solutions since exponential decays.

\subsection{Proof of Theorem \ref{THM-OPTIMAL-DECAY}}
Let us firstly recall the optimal estimates proposed in \cite[Propositions 3.1-3.3]{Ikehata=2021} as follows:
\begin{align}\label{Optimal-Kernel}
\left\|\frac{\sin(|\xi|^2t)}{|\xi|^2}\mathrm{e}^{-c|\xi|^4t}\right\|_{L^2}\simeq\ml{D}_n(t) \ \ \mbox{and}\ \ 
\left\||\xi|^k\mathrm{e}^{-c|\xi|^4t}\right\|_{L^2}\simeq t^{-\frac{k}{4}-\frac{n}{8}},
\end{align}
for any $k\in\mb{N}_0$ and $t\gg1$, where $\ml{D}_n(t)$ was introduced in \eqref{Decay-fun}. Then, by applying H\"older's inequality and the Hausdorff-Young inequality, we can obtain from \eqref{Est-04} and \eqref{PROP-LARGE-BDD} that
\begin{align*}
\|u(t,\cdot)\|_{L^2}
&\lesssim\left\|\chi_{\intt}(\xi)\mathrm{e}^{-c|\xi|^4t}\right\|_{L^2}\left(\|u_0\|_{L^1}+\|\theta_0\|_{L^1}\right)+\left\|\chi_{\intt}(\xi)\frac{\sin(|\xi|^2t)}{|\xi|^2}\mathrm{e}^{-c|\xi|^4t}\right\|_{L^2}\|u_1\|_{L^1}\\
&\quad+\mathrm{e}^{-ct}\left(\|u_0\|_{L^2}+\|u_1\|_{L^2}+\|\theta_0\|_{L^2}\right)\\
&\lesssim t^{-\frac{n}{8}}\left(\|u_0\|_{L^2\cap L^1}+\|\theta_0\|_{L^2\cap L^1}\right)+\ml{D}_n(t)\|u_1\|_{L^2\cap L^1}
\end{align*}
for large-time $t\gg1$, which gets the upper bound estimate.

 For the lower one, from \eqref{Est-05} and \eqref{Est-01}, we actually know
\begin{align*}
\|u(t,\cdot)-J_0(t,|D|)u_1(\cdot)\|_{L^2}&\lesssim\big\|\chi_{\intt}(\xi)\big(\widehat{u}(t,\xi)-\widehat{g}_1(t,\xi)\big)\big\|_{L^2}+\big\|\chi_{\intt}(\xi)\big(\widehat{g}_1(t,\xi)-\widehat{J}_0(t,|\xi|)\widehat{u}_1(\xi)\big)\big\|_{L^2}\\
&\quad+\big\|\big(1-\chi_{\intt}(\xi)\big)\widehat{u}(t,\xi)\big\|_{L^2}+\big\|\big(1-\chi_{\intt}(\xi)\big)\widehat{J}_0(t,|\xi|)\widehat{u}_1(\xi)\big\|_{L^2}\\
&\lesssim t^{-\frac{n}{8}}\|(u_0,u_1,\theta_0)\|_{(L^2\cap L^1)^3}
\end{align*}
for $t\gg1$. Here, \eqref{Optimal-Kernel} has been used.  For another thing, the decomposition
\begin{align*}
J_0(t,|D|)u_1(x)-J_0(t,x)P_{u_1}&=\int_{|y|\leqslant t^{\alpha_0}}\big(J_0(t,x-y)-J_0(t,x)\big)u_1(y)\mathrm{d}y\\
&\quad+\int_{|y|\geqslant t^{\alpha_0}}J_0(t,x-y)u_1(y)\mathrm{d}y-J_0(t,x)\int_{|y|\geqslant t^{\alpha_0}}u_1(y)\mathrm{d}y
\end{align*}
with a sufficiently small constant $\alpha_0>0$, as well as Taylor's expansion
\begin{align}\label{star-2}
|J_0(t,x-y)-J_0(t,x)|\lesssim|y|\,|\nabla J_0(t,x-\gamma_0y)|
\end{align}
with a constant $\gamma_0\in(0,1)$, show
\begin{align}\label{star-3}
\|J_0(t,|D|)u_1(\cdot)-J_0(t,\cdot)P_{u_1}\|_{L^2}&\lesssim t^{\alpha_0}\left\|\frac{\sin(|\xi|^2t)}{|\xi|}\mathrm{e}^{-c|\xi|^4t}\right\|_{L^2}\|u_1\|_{L^1}+\|\widehat{J}_0(t,|\xi|)\|_{L^2}\int_{|y|\geqslant t^{\alpha_0}}|u_1(y)|\mathrm{d}y\notag\\
&\lesssim t^{\alpha_0-\frac{1}{4}}\ml{D}_n(t)\|u_1\|_{L^1}+o\big(\ml{D}_n(t)\big)
\end{align}
for $t\gg1$, where we considered
\begin{align*}
\left\|\frac{\sin(|\xi|^2t)}{|\xi|}\mathrm{e}^{-c|\xi|^4t}\right\|_{L^2}\lesssim \left\||\xi|\mathrm{e}^{-\frac{c}{2}|\xi|^4t}\right\|_{L^{\infty}}\left\|\frac{\sin(|\xi|^2t)}{|\xi|^2}\mathrm{e}^{-\frac{c}{2}|\xi|^4t}\right\|_{L^2}\lesssim t^{-\frac{1}{4}}\ml{D}_n(t),
\end{align*}
and the fact that $u_1\in L^1$ associates with
\begin{align*}
\lim\limits_{t\to\infty}\int_{|y|\geqslant t^{\alpha_0}}|u_1(y)|\mathrm{d}y=0.
\end{align*}
In conclusion, an application of Minkowski's inequality immediately implies
\begin{align*}
\|u(t,\cdot)\|_{L^2}&\geqslant\|J_0(t,\cdot)\|_{L^2}|P_{u_1}|-\|u(t,\cdot)-J_0(t,\cdot)P_{u_1}\|_{L^2}\\
&\gtrsim\ml{D}_n(t)|P_{u_1}|-\|u(t,\cdot)-J_0(t,|D|)u_1(\cdot)\|_{L^2}-\|J_0(t,|D|)u_1(\cdot)-J_0(t,\cdot)P_{u_1}\|_{L^2}\\
&\gtrsim\ml{D}_n(t)|P_{u_1}|-t^{-\frac{n}{8}}\|(u_0,u_1,\theta_0)\|_{(L^2\cap L^1)^3}-t^{\alpha_0-\frac{1}{4}}\ml{D}_n(t)\|u_1\|_{L^1}-o\big(\ml{D}_n(t)\big)\\
&\gtrsim\ml{D}_n(t)|P_{u_1}|
\end{align*}
for $t\gg1$ by taking $\alpha_0$ sufficiently small. So, our proof of Theorem \ref{THM-OPTIMAL-DECAY} is complete.

\subsection{Proof of Theorem \ref{THM-OPTIMAL-LEADING}}
Recalling the profile defined in \eqref{Profile}, we indeed may find the next decomposition:
\begin{align*}
u(t,x)-\varphi(t,x)=\ml{E}_0(t,x)+\ml{E}_1(t,x)+\ml{E}_2(t,x),
\end{align*}
with the components
\begin{align*}
\ml{E}_0(t,x)&:=u(t,x)-\big(J_0(t,|D|)+H(t,|D|)\big)u_1(x)-J_1(t,|D|)u_0(x)+\sigma^{-1}\Delta J_0(t,|D|)\theta_0(x),\\
\ml{E}_1(t,x)&:=J_0(t,|D|)u_1(x)-J_0(t,x)P_{u_1}-\left\langle\nabla J_0(t,x), M_{u_1}\right\rangle,\\
\ml{E}_2(t,x)&:=\big(H(t,|D|)u_1(x)-H(t,x)P_{u_1}\big)+\big(J_1(t,|D|)u_0(x)-J_1(t,x)P_{u_0}\big)\\
&\ \quad-\sigma^{-1}\big(\Delta J_0(t,|D|)\theta_0(x)-\Delta J_0(t,x)P_{\theta_0}\big).
\end{align*}
First of all, we employ the derived estimates \eqref{star-1}, \eqref{Est-02} as well as \eqref{Est-03} to get
\begin{align*}
\|\ml{E}_0(t,\cdot)\|_{L^2}&\lesssim\left\|\chi_{\intt}(\xi)|\xi|^2\mathrm{e}^{-c|\xi|^4t}\left(|\widehat{u}_0(\xi)|+|\widehat{u}_1(\xi)|+|\widehat{\theta}_0(\xi)|\right)\right\|_{L^2}+\left\|\big(1-\chi_{\intt}(\xi)\big)\widehat{\ml{E}}_0(t,\xi)\right\|_{L^2}\\
&\lesssim t^{-\frac{1}{2}-\frac{n}{8}}\|(u_0,u_1,\theta_0)\|_{(L^2\cap L^1)^3}
\end{align*}
for $t\gg1$. To treat the second error term, we re-formulate it by
\begin{align*}
\ml{E}_1(t,x)&=\int_{|y|\leqslant t^{\alpha_1}}\big(J_0(t,x-y)-J_0(t,x)-\left\langle y,\nabla J_0(t,x)\right\rangle\big)u_1(y)\mathrm{d}y\\
&\quad+\int_{|y|\geqslant t^{\alpha_1}}\big(J_0(t,x-y)-J_0(t,x)\big)u_1(y)\mathrm{d}y+\int_{|y|\geqslant t^{\alpha_1}}\left\langle-y,\nabla J_0(t,x)\right\rangle u_1(y)\mathrm{d}y,
\end{align*}
equipping a sufficiently small constant $\alpha_1>0$. Again from \eqref{star-2} and 
\begin{align*}
|J_0(t,x-y)-J_0(t,x)-\left\langle y,\nabla J_0(t,x)\right\rangle|\lesssim |y|^2|\Delta J_0(t,x-\gamma_1y)|
\end{align*}
with $\gamma_1\in(0,1)$, one derives
\begin{align*}
\|\ml{E}_1(t,\cdot)\|_{L^2}&\lesssim t^{2\alpha_1}\|\,|\xi|^2\widehat{J}_0(t,|\xi|)\|_{L^2}\|u_1\|_{L^1}+\|\,|\xi|\widehat{J}_0(t,|\xi|)\|_{L^2}\int_{|y|\geqslant t^{\alpha_1}}|y||u_1(y)|\mathrm{d}y\\
&\lesssim t^{2\alpha_1-\frac{n}{8}}\|u_1\|_{L^1}+o\big(\ml{B}_n(t)\big)
\end{align*}
for $t\gg1$, where we used \cite[Lemma 4.5]{Ikehata=2021} and the weighted assumption $u_1\in L^{1,1}$. Afterward, similarly to \eqref{star-3}, we may calculate
\begin{align*}
\|\ml{E}_2(t,\cdot)\|_{L^2}=o(t^{-\frac{n}{8}})
\end{align*}
for $t\gg1$ also, which can be guaranteed by the hypothesis $u_0,u_1,\theta_0\in L^{1,1}$. Summing up all obtained estimates from the previous statements, we assert that
\begin{align*}
\|u(t,\cdot)-\varphi(t,\cdot)\|_{L^2}\leqslant\|\ml{E}_0(t,\cdot)\|_{L^2}+\|\ml{E}_1(t,\cdot)\|_{L^2}+\|\ml{E}_2(t,\cdot)\|_{L^2}=o\big(\ml{B}_n(t)\big)
\end{align*}
for large-time $t\gg1$, and our proof is complete now.

\section*{Acknowledgments}
The author thanks Ryo Ikehata (Hiroshima University)  for some suggestions in the paper.

\end{document}